\documentclass[11pt,a4paper]{article}

\usepackage{amsmath}
\usepackage{amssymb,verbatim}
\usepackage{amsfonts}
\usepackage{amsthm}
\usepackage{pifont}
\usepackage{graphicx}

\newtheorem{theorem}{Theorem}

\newtheorem{corollary}{Corollary}

\theoremstyle{definition}
\newtheorem*{remark}{Remark}

\newcommand{\ds}{\displaystyle}
\newcommand{\beq}{\begin{equation*}}
\newcommand{\eeq}{\end{equation*}}
\newcommand{\beqn}{\begin{equation}}
\newcommand{\eeqn}{\end{equation}}

\newcommand{\M}{\mathcal{S}(g,\:\!s)}
\newcommand{\RR}{\mathbb R}

\newcommand{\dd}{\mathrm{d}}
\newcommand{\n}{\natural}
\newcommand{\T}{\mathcal{T}_{a,\;\!b}}
\newcommand{\R}{\mathcal{R}_{a,\;\!b}}
\newcommand{\A}{\tilde{A}}
\newcommand{\gR}{g_\mathcal{R}}
\newcommand{\GR}{G_\mathcal{R}}
\newcommand{\AR}{A_\mathcal{R}}
\newcommand{\uR}{u_\mathcal{R}}
\newcommand{\JR}{J_\mathcal{R}}

\begin{document}

\title{The distribution function of the distance between two random points in a right-angled triangle}
\author{Uwe B\"asel}
\date{}
\maketitle

\begin{abstract}
\noindent In this paper we obtain the density function and the distribution function of the distance between two uniformly and independently distributed random points in any right-angled triangle. The density function is derived from the chord length distribution function using Piefke's formula. We conclude results for random distances between two congruent right-angled triangles together forming a rectangle.\\[0.2cm]
\end{abstract}

\section{Introduction}

A random line $g$ intersecting a convex set $\mathcal{K}$ in the plane produces a chord of $\mathcal{K}$. The length $s$ of this chord is a random variable. The chord length distribution function of a regular triangle was calculated by Sulanke \cite[p.~57]{Sulanke}. Harutyunyan and Ohanyan \cite{HO} calculated the chord length distribution function for regular polygons (see also B\"asel \cite[pp.\ 2-9]{Baesel}).

The distance $t$ between two points chosen independently and uniformly at random from $\mathcal{K}$ is also a random variable. Borel \cite{Borel} considered random distances in elementary geometric figures such as triangles, squares and so on (see \cite[p.\ 163]{MMP}). The expectations of the distance between two random points for an equilateral triangle and a rectangle are to be found in \cite[p.~49]{Santalo}. Ghosh \cite{Ghosh} derived the distance density function for a rectangle. B\"asel \cite[pp.\ 10-16]{Baesel} derived the density function and the distribution function of the distance for regular polygons from its chord length distribution function using Piefke's formula \cite[p.\ 130]{Piefke}. There are a lot of results concerning the distance $t$  within a convex set or in two convex sets, see Chapter 2 in \cite{Mathai}.

For practical applications in physics, material sciences, operations research, communication networks and other fields see \cite{Gille}, \cite{Marsaglia} and \cite{ZP}.

The first aim of the present paper is to derive the chord length distribution function for any right-angled triangle. The second aim is to conclude the distribution function of the distance between two points chosen independently and uniformly from a right-angled triangle.  

$A=(0,b)$, $B=(a,0)$ and $C=(0,0)$ are the vertices of the right-angled triangle (see Fig.\:\ref{Bild1}b). We assume $|\overline{AC}|=b\geq a=|\overline{BC}|$ and denote this triangle by $\T$. Clearly, it is possible to bring every right-angled triangle in such a position by use of Euclidean motions (translations and rotations) and, if necessary, a reflection. These operations do not influence the chord length distribution and the point distance distribution. $c=|\overline{AB}|$ is the length of the hypotenuse and $\alpha=\arctan(a/b)$ the inner angle between $\overline{AB}$ and $\overline{AC}$. $u=a+b+c$ is the length of the perimeter of $\T$, $\A=ab/2$ its area and $h=a\cos\alpha$ the height over $\overline{AB}$. The hypotenuse $\overline{AB}$ of the triangle $\T$ is the diameter of its circumcircle (converse of Thales' theorem).

\section{Chord length distribution function}

A straight line $g$ in the plane is determined by the angle $\psi$, $0\leq\psi<2\pi$, that the direction perpendicular to $g$ makes with a fixed direction (e.g. the $x$-axis) and by its distance $p$, $0\leq p<\infty$, from the origin $O$:
\[ g = g(p,\psi) = \{(x,y)\in\RR^2:\, x\cos\psi+y\sin\psi = p\}\,. \]
The measure $\mu$ of a set of lines $g(p,\psi)$ is defined by the integral, over the set, of the differential form $\dd g=\dd p\:\dd\psi$. Up to a constant factor, this measure is the only one that is invariant under motions in the Euclidean plane \cite[p.~28]{Santalo}.

We use the chord length distribution function of $\T$ in the form
\beq
  F(s) = \frac{1}{u}\:\mu(\{g:\,g\cap\T\not=\emptyset,\,|\chi(g)|\leq s\})\,,
\eeq
where $\chi(g)=g\cap\T$ is the chord of $\T$, produced by the line $g$, and $|\chi(g)|$ the length of $\chi(g)$ (see  \cite[p.~863]{Gates} and \cite[p.\ 161]{AMS}). (The measure of all lines $g$ that intersect a convex set is equal to its perimeter \cite[p.\ 30]{Santalo}.) So it remains to calculate the measure of all lines that produce a chord of length $|\chi(g)|\leq s$\,. Using the abbreviation
\beq
  \M:=\{g:\,g\cap\T\not=\emptyset,\,|\chi(g)|\leq s\}\,,
\eeq
we have
\beq
 \mu(\M) = \int_{\M}\,\dd g = \int_{\M}\,\dd p\:\dd\psi\,.
\eeq
Instead of the angle $\psi$ it is possible to use the angle $\phi$ that a line makes with the positive $x$-axis, hence
\beq
 \mu(\M) = \int_{\M}\,\dd p\:\dd\phi\,.
\eeq

\begin{theorem} \label{chord-distribution}
The chord length distribution function $F$ of $\T$ is given by
\beq
  F(s) = \left\{\begin{array}{lcl}
	0 & \mbox{if} & s<0\,,\\[0.1cm]
	H_1(s)/u & \mbox{if} & 0\leq s<h\,,\\[0.1cm]
	H_2(s)/u & \mbox{if} & h\leq s<a\,,\\[0.1cm]
	H_3(s)/u & \mbox{if} & a\leq s<b\,,\\[0.1cm]
	H_4(s)/u & \mbox{if} & b\leq s<c\,,\\[0.1cm]
	1 & \mbox{if} & s\geq c\,,
  \end{array}\right.
\eeq
where
\begin{align*}
  H_1(s) = {} & \frac{\Theta_1\,s}{4}\,,\\
  H_2(s) = {} & \frac{\Theta_2\,s}{4}+\Theta_3\,L_1(s,h)+c\,L_2(s,h)\,,\displaybreak[0]\\
  H_3(s) = {} & a+\Theta_4\,s+\frac{\Theta_3}{2}\,L_1(s,h)+\frac{c}{2}\,L_2(s,h)
					+\frac{b}{2a}\,L_1(s,a)+\frac{b}{2}\,L_2(s,a)\,,\displaybreak[0]\\
  H_4(s) = {} & \Theta_5+\frac{\Theta_6\,s}{4}+\frac{b}{2a}\,L_1(s,a)+\frac{b}{2}\,L_2(s,a)
					+\frac{a}{2b}\,L_1(s,b)+\frac{a}{2}\,L_2(s,b)
\end{align*}
with
\begin{align*}
  \Theta_1 = {} & \frac{a}{b}\,(2\alpha+\pi)+\frac{2b}{a}\,(\pi-\alpha)+6 \;,\;\;
  \Theta_2 = \frac{a}{b}\,(2\alpha-\pi)-\frac{2b}{a}\,\alpha+6\,,\\
  \Theta_3 = {} & \frac{a}{b}+\frac{b}{a} \;,\;\;
  \Theta_4 = 1-\frac{b}{a}\,\alpha \;,\;\;
  \Theta_5 = a+b \;,\;\;
  \Theta_6 = \frac{a}{b}\,(2\alpha-\pi)-\frac{2b}{a}\,\alpha+2
\end{align*}
and
\beq
  L_1(s,a) = s\arcsin\frac{a}{s} \;,\quad L_2(s,a) = \sqrt{1-\left(\frac{a}{s}\right)^2}\;.  					
\eeq
\end{theorem}

\begin{proof}
In the following we need the width $w$ of $\T$ in the direction perpendicular to the angle $\phi$ that is given by:
\beq
  w(\phi) = \left\{\begin{array}{lcl}
	w_1(\phi) & \mbox{if} & 0\leq\phi<\pi/2\,,\\[0.1cm]
	w_2(\phi) & \mbox{if} & \pi/2\leq\phi<\pi/2+\alpha\,,\\[0.1cm]
	w_3(\phi) & \mbox{if} & \pi/2+\alpha\leq\phi\leq\pi\,,  
  \end{array}\right.
\eeq
where
\beq
  w_1(\phi) = a\,\frac{\cos(\phi-\alpha)}{\sin\alpha} \;,\quad
  w_2(\phi) = a\sin\phi \;,\quad
  w_3(\phi) = -a\,\frac{\cos\phi}{\tan\alpha}\;.
\eeq
We have to distinguish the cases $0\leq s<h$, $h\leq s<a$, $a\leq s<b$ and $b\leq s<c$, and put
\beq
  \begin{array}{llc@{\;,\quad}llc}
	H_1(s) = \mu(\M) & \mbox{if} & 0\leq s<h &
	H_2(s) = \mu(\M) & \mbox{if} & h\leq s<a\;,\\[0.1cm]
	H_3(s) = \mu(\M) & \mbox{if} & a\leq s<b &
	H_4(s) = \mu(\M) & \mbox{if} & b\leq s<c\;.
  \end{array}
\eeq
\newpage\noindent
We demonstrate the calculation of $H_2$. Here it is necessary to distinguish the following intervals of $\phi$ (see Fig.\ \ref{Bild1})\,:
\begin{list}{\labelitemi}{
\setlength{\leftmargin}{0.6cm}
\setlength{\itemindent}{-0.35cm}
}
\item[a)] $0\leq\phi<\phi_1=\alpha-\arccos\left((a/s)\cos\alpha\right)$\,: All lines $g$ with fixed angle $\phi$ produce a chord $\chi$ of length $|\chi(g)|\leq s$ if they are lying in one of the two strips of breadth   
\beq
  b_1(s,\phi) = \frac{s}{2}\,\frac{\sin(2\phi-\alpha)+\sin\alpha}{\cos\alpha} \;\;\mbox{and}\;\;
  b_2(s,\phi) = \frac{s}{2}\,\frac{\cos(2\phi-\alpha)+\cos\alpha}{\sin\alpha}
\eeq
respectively.
\item[b)] $\phi_1\leq\phi<\phi_2=\alpha+\arccos\left((a/s)\cos\alpha\right)$\,: All lines $g$ with fixed angle $\phi$ and $g\cap\T\not=\emptyset$ produce a chord of length $|\chi(g)|\leq s$. These lines are lying in a strip of breadth $w_1(\phi)$.
\item[c)] $\phi_2\leq\phi<\pi/2$\,: See $0\leq\phi<\phi_1$.    
\item[d)] $\pi/2\leq\phi<\pi/2+\alpha$\,: All lines $g$ with fixed angle $\phi$ produce a chord of $\mbox{length}\leq s$ if they are lying in one of the two strips of breadth $b_1(s,\phi)$ and   
\beq
  b_3(s,\phi) = -\frac{s}{2}\,\sin 2\phi
\eeq
respectively.
\item[e)] $\pi/2+\alpha\leq\phi<\pi$\,: All lines $g$ with fixed angle $\phi$, lying in one of the two strips of breadth $b_2(s,\phi)$ and $b_3(s,\phi)$ repectively, produce a chord of length $\leq s$.
\end{list}
\begin{figure}[h]
  \vspace{-0.4cm}
  \begin{center}
    \includegraphics[scale=1]{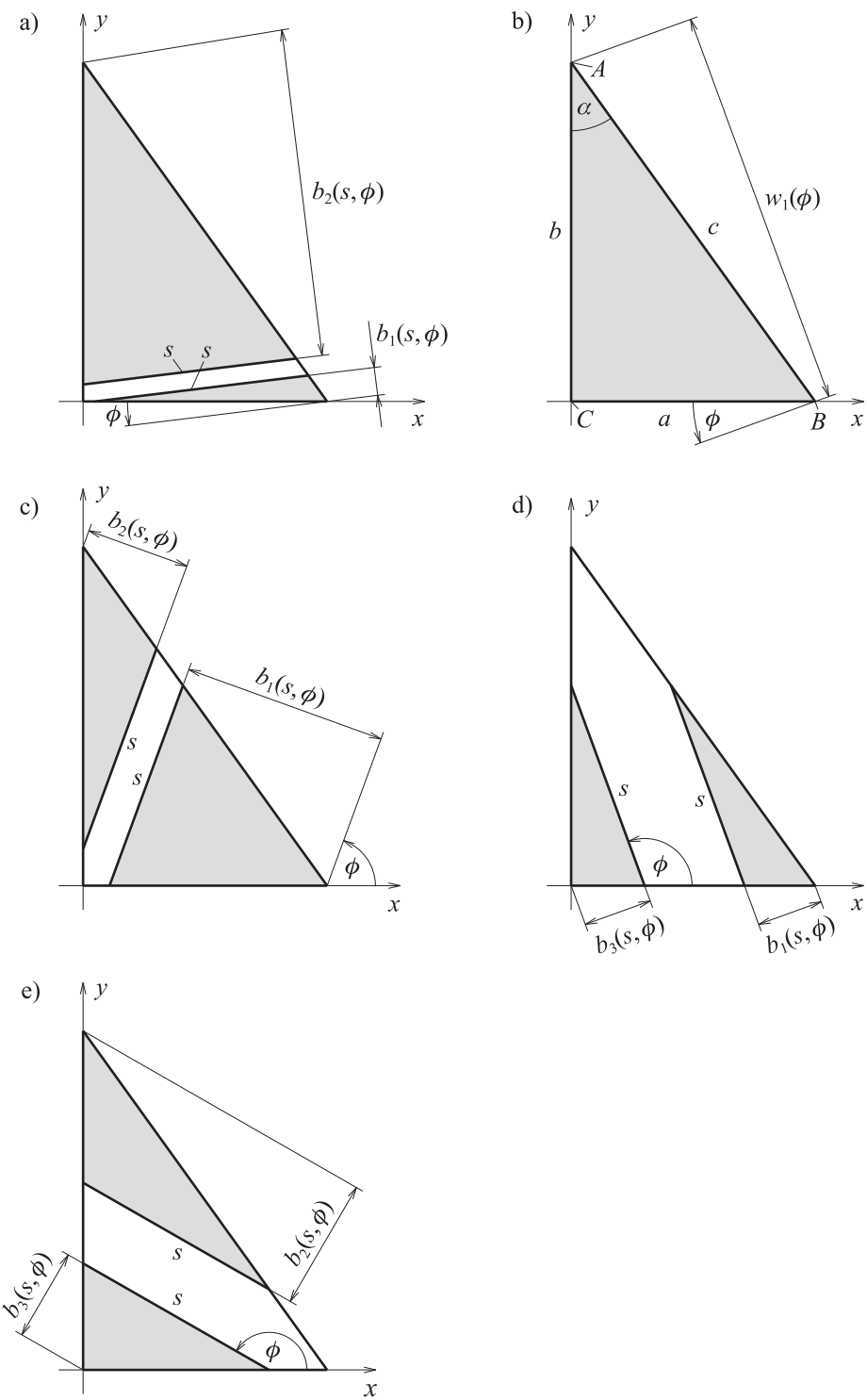}
  \end{center}
  \vspace{-0.3cm}
  \caption{\label{Bild1} Calculation of $H_2$}
\end{figure}
Therefore, using the abbreviations $b_i$ and $w_i$ for $b_i(s,\phi)$ and $w_i(\phi)$ respectively,
\begin{align*}
  H_2(s) = \int_{\M}\,\dd p\:\dd\phi 
		 = {} & \int_0^{\phi_1}(b_1+b_2)\,\dd\phi
				+ \int_{\phi_1}^{\phi_2}w_1\,\dd\phi
				+ \int_{\phi_2}^{\pi/2}(b_1+b_2)\,\dd\phi\\
			  &	+ \int_{\pi/2}^{\pi/2+\alpha}(b_1+b_3)\,\dd\phi
				+ \int_{\pi/2+\alpha}^\pi(b_2+b_3)\,\dd\phi\,.
\end{align*}
With the angles 
\beq
  \phi_3 = \pi-\arccos\dfrac{a}{s} \quad\mbox{and}\quad
  \phi_4 = \pi-\arcsin\left(\dfrac{a}{s}\,\cot\alpha\right)
\eeq
we analogously get
\begin{align*}
  H_1(s) = {} & \int_0^{\pi/2}(b_1+b_2)\,\dd\phi
				+ \int_{\pi/2}^{\pi/2+\alpha}(b_1+b_3)\,\dd\phi
				+ \int_{\pi/2+\alpha}^\pi(b_2+b_3)\,\dd\phi\,,\displaybreak[0]\\
  H_3(s) = {} & \int_0^{\phi_2}w_1\,\dd\phi
				+ \int_{\phi_2}^{\pi/2}(b_1+b_2)\,\dd\phi
			  	+ \int_{\pi/2}^{\pi/2+\alpha}(b_1+b_3)\,\dd\phi\\
			  &	+ \int_{\pi/2+\alpha}^{\phi_3}(b_2+b_3)\,\dd\phi
				+ \int_{\phi_3}^\pi w_3\,\dd\phi\,,
\end{align*}
\clearpage
\begin{align*}				
  H_4(s) = {} & \int_0^{\pi/2}w_1\,\dd\phi
				+ \int_{\pi/2}^{\phi_4}w_2\,\dd\phi
				+ \int_{\phi_4}^{\pi/2+\alpha}(b_1+b_3)\,\dd\phi\\
			  &	+ \int_{\pi/2+\alpha}^{\phi_3}(b_2+b_3)\,\dd\phi
				+ \int_{\phi_3}^\pi w_3\,\dd\phi\,.	
\end{align*}
Using the indefinite integrals
\begin{align*}
  \int b_1(s,\phi)\,\dd\phi = {} & \frac{s}{4}\,\frac{2\phi\sin\alpha-\cos(2\phi-\alpha)}{\cos\alpha}\,,\\[0.1cm]
  \int b_2(s,\phi)\,\dd\phi = {} & \frac{s}{4}\,\frac{2\phi\cos\alpha+\sin(2\phi-\alpha)}{\sin\alpha}\,,\\[0.1cm]
  \int b_3(s,\phi)\,\dd\phi = {} & \frac{s}{2}\,\cos^2\phi\,,\\[0.1cm]
  \int w_1(\phi)\,\dd\phi\:\, = {} & \frac{a\sin(\phi-\alpha)}{\sin\alpha} = a(\cot\alpha\sin\phi-\cos\phi)\,,\\[0.1cm]
  \int w_2(\phi)\,\dd\phi\:\, = {} & -a\cos\phi\;,\quad
  \int w_3(\phi)\,\dd\phi   = -\frac{a\sin\phi}{\tan\alpha}\,,
\end{align*}
one finds the result of Theorem \ref{chord-distribution}.
\end{proof}
\begin{remark}
It is easy to derive the chord length density function from Theorem \ref{chord-distribution}. 
\end{remark}

\section{Point distances}

\begin{theorem} \label{distance-density}
The density function $g$ of the distance $t$ between two random points in $\T$ is given by
\beq
  g(t) = \left\{\begin{array}{ll}
	\dfrac{2t}{\A}\left[\pi+\dfrac{1}{\A}\,(J^*(t)-ut)\right] & \mbox{if} \quad 0\leq t\leq c\,,\\[0.4cm]
	0 & \mbox{else}\,,
  \end{array}\right.
\eeq
where
\beq
  J^*(t) = \left\{\begin{array}{lcl}
	J_1^*(0,t) & \mbox{if} & 0\leq t<h\,,\\[0.1cm]
	J_1^*(0,h)+J_2^*(h,t) & \mbox{if} & h\leq t<a\,,\\[0.1cm]
	J_1^*(0,h)+J_2^*(h,a)+J_3^*(a,t) & \mbox{if} & a\leq t<b\,,\\[0.1cm]
	J_1^*(0,h)+J_2^*(h,a)+J_3^*(a,b)+J_4^*(b,t) & \mbox{if} & b\leq t\leq c
  \end{array}\right.
\eeq
with
\beq
  J_k^*(s,t) = H_k^*(t) - H_k^*(s)\,,
\eeq
where
\begin{align*}
  H_1^*(t) = {} & \frac{\Theta_1\,t^2}{8}\,,\displaybreak[0]\\[0.1cm]
  H_2^*(t) = {} & \frac{\Theta_2\,t^2}{8}+\Theta_3\,L_1^*(t,h)+c\,L_2^*(t,h)\,,\\[0.1cm]
  H_3^*(t) = {} & at+\frac{\Theta_4\,t^2}{2}+\frac{\Theta_3}{2}\,L_1^*(t,h)+\frac{c}{2}\,L_2^*(t,h)
					+\frac{b}{2a}\,L_1^*(t,a)+\frac{b}{2}\,L_2^*(t,a)\,,\\[0.1cm]
  H_4^*(t) = {} & \Theta_5\,t+\frac{\Theta_6\,t^2}{8}+\frac{b}{2a}\,L_1^*(t,a)+\frac{b}{2}\,L_2^*(t,a)
					+\frac{a}{2b}\,L_1^*(t,b)+\frac{a}{2}\,L_2^*(t,b)
\end{align*}
with
\begin{align*}
  L_1^*(t,a) = {} & \frac{1}{2}\left(a\sqrt{t^2-a^2}+t^2\arcsin\frac{a}{t}\right) \,,\;\;
  L_2^*(t,a) = \sqrt{t^2-a^2}+a\arcsin\frac{a}{t}\,.  					
\end{align*}
\end{theorem}

\begin{proof}
According to a theorem of Piefke \cite[p.\ 130]{Piefke}, the density function of the distance between two random points is given by
\beq
  g(t) = \frac{2ut}{\A^2}\int_t^c (s-t)f(s)\,\dd s\,,
\eeq
where $f$ is the density function of the chord length. Following \cite[pp.\ 11/12]{Baesel}, we obtain
\beq
  g(t)
	= \frac{2t}{\A}\left[\pi-\frac{u}{\A}\left(t-\int_0^tF(s)\,\dd s\right)\right]
	= \frac{2t}{\A}\left[\pi+\frac{1}{\A}\,(J^*(t)-ut)\right]
\eeq
with
\beq
  J^*(t):=u\int_0^t F(s)\,\dd s\,.
\eeq
This yields for $0\leq t<h$
\beq
  J^*(t) = \int_0^t H_1(s)\,\dd s = H_1^*(t)-H_1^*(0) = J_1^*(0,t)\,,
\eeq
and for $h\leq t<a$
\begin{align*}
  J^*(t)
	 = {} & u\int_0^t F(s)\,\dd s = \int_0^h H_1(s)\,\dd s+\int_h^t H_2(s)\,\dd s\\[0.1cm]
	 = {} & H_1^*(h)-H_1^*(0)+H_2^*(t)-H_2^*(h) = J_1^*(0,h)+J_2^*(h,t)\,.
\end{align*}
The calculations of $J^*(t)$ for the intervals $a\leq t<b$ and $b\leq t<c$ can be carried out in the same manner.  
 
Now we determine the indefinite integrals of $L_1(t,a)$ and $L_2(t,a)$. Using integration by parts, we find
\begin{align*}
  L_1^*(t,a) = {} & \int L_1(t,a)\,\dd t = \int t\arcsin\frac{a}{t}\;\dd t\\
			 = {} & \frac{1}{2}\left(t^2\arcsin\frac{a}{t}+a\int\frac{t}{\sqrt{t^2-a^2}}\;\dd t\right).
\end{align*}
Clearly,
\beq
  \int\frac{t}{\sqrt{t^2-a^2}}\;\dd t = \sqrt{t^2-a^2}  
\eeq
and therefore
\beq
  L_1^*(t,a) = \frac{1}{2}\left(a\sqrt{t^2-a^2}+t^2\arcsin\frac{a}{t}\right).
\eeq
Since $t\geq a>0$ in the present cases, 
\begin{align*}
  L_2^*(t,a) = {} & \int L_2(t,a)\,\dd t = \int\sqrt{1-\left(\frac{a}{t}\right)^2}\;\dd t
	 = \int\frac{\sqrt{t^2-a^2}}{t}\;\dd t\,,
\end{align*}
and (cf.\ \cite[p.\ 48, Eq.\ 217]{Bronstein})
\beq
  L_2^*(t,a) = \sqrt{t^2-a^2}+a\arcsin\frac{a}{t}\,. \qedhere
\eeq 
\end{proof}

\begin{corollary} \label{G}
The distribution function $G$ of the distance $t$ between two random points in $\T$ is given by
\beq
  G(t) = \left\{\begin{array}{lcl}
	0 & \mbox{if} & t<0\,,\\[0.2cm]
	\dfrac{1}{\A}\left[t^2\left(\pi-\dfrac{2u}{3\A}\,t\right)+\dfrac{2}{\A}\,J^\n(t)\right] & \mbox{if} & 0\leq t<c\,,\\[0.4cm]
	1 & \mbox{if} & t\geq c
  \end{array}\right.
\eeq
with
\beq
  J^\n(t) = \left\{\begin{array}{lcl}
	K_1(t) & \mbox{if} & 0\leq t<h\,,\\[0.1cm]
	K_1(h)+K_2(t) & \mbox{if} & h\leq t<a\,,\\[0.1cm]
	K_1(h)+K_2(a)+K_3(t) & \mbox{if} & a\leq t<b\,,\\[0.1cm]
	K_1(h)+K_2(a)+K_3(b)+K_4(t) & \mbox{if} & b\leq t<c\,,
  \end{array}\right.
\eeq
where
\begin{align*}
  K_1(t) = {} & J_1^\n(0,t)\,,\\[0.2cm]
  K_2(t) = {} & \frac{1}{2}\left(t^2-h^2\right)[J_1^*(0,h)-H_2^*(h)]+J_2^\n(h,t)\,,\\[0.1cm]
  K_3(t) = {} & \frac{1}{2}\left(t^2-a^2\right)[J_1^*(0,h)+J_2^*(h,a)-H_3^*(a)]+J_3^\n(a,t)\,,\\[0.1cm]
  K_4(t) = {} & \frac{1}{2}\left(t^2-b^2\right)[J_1^*(0,h)+J_2^*(h,a)+J_3^*(a,b)-H_4^*(b)]+J_4^\n(b,t)
\end{align*}
with $J_k^*$ and $H_k^*$ according to Theorem \ref{distance-density} and 
\beq
  J_k^\n(s,t) = H_k^\n(t) - H_k^\n(s)\,,
\eeq
where
\begin{align*}
  H_1^\n(t) = {} & \frac{\Theta_1\,t^4}{32}\,,\\[0.1cm]
  H_2^\n(t) = {} & \frac{\Theta_2\,t^4}{32}+\Theta_3\,L_1^\n(t,h)+c\,L_2^\n(t,h)\,,\\[0.1cm]
  H_3^\n(t) = {} & \frac{at^3}{3}+\frac{\Theta_4\,t^4}{8}
					+\frac{\Theta_3}{2}\,L_1^\n(t,h)+\frac{c}{2}\,L_2^\n(t,h)
					+\frac{b}{2a}\,L_1^\n(t,a)+\frac{b}{2}\,L_2^\n(t,a)\,,\\[0.1cm]
  H_4^\n(t) = {} & \frac{\Theta_5\,t^3}{3}+\frac{\Theta_6\,t^4}{32}
					+\frac{b}{2a}\,L_1^\n(t,a)+\frac{b}{2}\,L_2^\n(t,a)
					+\frac{a}{2b}\,L_1^\n(t,b)+\frac{a}{2}\,L_2^\n(t,b)
\end{align*}
with
\begin{align*}
  L_1^\n(t,a) = {} & \frac{1}{8}\left[\frac{5a}{3}\left(t^2-a^2\right)^{3/2}
						+a^3\,\sqrt{t^2-a^2}+t^4\arcsin\frac{a}{t}\right]\,,\\[0.2cm]
  L_2^\n(t,a) = {} & \frac{1}{3}\left(t^2-a^2\right)^{3/2}+\frac{a}{2}\left(a\sqrt{t^2-a^2}
						+t^2\,\arcsin\frac{a}{t}\right)\,.  					
\end{align*}
\end{corollary}

\begin{proof}
For $0\leq t<c$ one gets
\begin{align*}
  G(t) 
	= {} & \int_0^t g(\tau)\,\dd\tau 
	= \int_0^t\left(\frac{2\pi\tau}{\A}-\frac{2u\tau^2}{\A^2}+\frac{2u\tau}{\A^2}\int_0^\tau F(s)\,\dd s\right)\dd\tau\\
	= {} & \frac{\pi t^2}{\A}-\frac{2ut^3}{3\A^2}+\frac{2u}{\A^2}\int_0^t\tau\left(\int_0^\tau F(s)\,\dd s\right)\dd\tau\\
	= {} & \frac{\pi t^2}{\A}-\frac{2ut^3}{3\A^2}+\frac{2}{\A^2}\int_0^t\tau J^*(\tau)\,\dd\tau
	= \frac{1}{\A}\left[t^2\left(\pi-\dfrac{2u}{3\A}\,t\right)+\dfrac{2}{\A}\,J^\n(t)\right]
\end{align*}
with
\beq
  J^\n(t) := \int_0^t sJ^*(s)\,\dd s\,.
\eeq 
It remains to calculate $J^\n(t)$. For $0\leq t<h$ we have
\begin{align*}
  J^\n(t) = {} & \int_0^t sJ_1^*(0,s)\,\dd s = \int_0^t s\,[H_1^*(s)-H_1^*(0)]\,\dd s\\[0.1cm]
		  = {} & H_1^\n(t)-H_1^\n(0)-\frac{1}{2}\,t^2H_1^*(0)
		  = H_1^\n(t)-H_1^\n(0)-0 = J_1^\n(0,t) = K_1(t)\,.
\end{align*}
For $h\leq t<a$ we find
\begin{align*}
  J^\n(t) = {} & K_1(h)+\int_h^t s\,[J_1^*(0,h)+J_2^*(h,s)]\,\dd s\\
		  = {} & K_1(h)+\frac{1}{2}\left(t^2-h^2\right)J_1^*(0,h)+\int_h^t sJ_2^*(h,s)\,\dd s\,.
\end{align*}
With
\begin{align*}
  \int_h^t sJ_2^*(h,s)\,\dd s
	= {} & \int_h^t s\,[H_2^*(s)-H_2^*(h)]\,\dd s 
	= \int_h^t sH_2^*(s)\,\dd s - H_2^*(h)\int_h^t s\,\dd s\displaybreak[0]\\
	= {} & H_2^\n(t)-H_2^\n(h)-\frac{1}{2}\left(t^2-h^2\right)H_2^*(h)\\
	= {} & J_2^\n(h,t)-\frac{1}{2}\left(t^2-h^2\right)H_2^*(h)
\end{align*} 
it follows that
\beq
  J^\n(t) = K_1(h)+\frac{1}{2}\left(t^2-h^2\right)[J_1^*(0,h)-H_2^*(h)]+J_2^\n(h,t) = K_1(h)+K_2(t) \,.	
\eeq
The calculation applies analogously to the intervals $a\leq t<b$ and $b\leq t<~\!\!c$.

Now we calculate the indefinite integrals of $L_1^*(t,a)$ and $L_2^*(t,a)$. We have
\begin{align*}
  L_1^\n(t,a)
	= {} & \int tL_1^*(t,a)\,\dd t = \frac{1}{2}\int t\left(a\sqrt{t^2-a^2}+t^2\arcsin\frac{a}{t}\right)\dd t\\
	= {} & \frac{1}{2}\left(a\int t\,\sqrt{t^2-a^2}\;\dd t+\int t^3\arcsin\frac{a}{t}\;\dd t\right)
\end{align*}
with
\beq
  \int t\,\sqrt{t^2-a^2}\;\dd t = \frac{1}{3}\left(t^2-a^2\right)^{3/2} \quad\mbox{\cite[p.\ 47, Eq.\ 214]{Bronstein}}\,.
\eeq
Using integration by parts, we find
\beq
  \int t^3\arcsin\frac{a}{t}\;\dd t
	= \frac{1}{4}\left(t^4\arcsin\frac{a}{t}+a\int\frac{t^2}{\sqrt{1-(a/t)^2}}\;\dd t\right).
\eeq
Since $t\geq a>0$ in the present cases, 
\beq
  \int\frac{t^2}{\sqrt{1-(a/t)^2}}\;\dd t = \int\frac{t^3}{\sqrt{t^2-a^2}}\;\dd t\,, 
\eeq
and
\beq
  \int\frac{t^3}{\sqrt{t^2-a^2}}\;\dd t
	= \frac{1}{3}\left(t^2-a^2\right)^{3/2}+a^2\,\sqrt{t^2-a^2} \quad\mbox{\cite[p.\ 48, Eq.\ 223]{Bronstein}}\,.
\eeq
This yields
\beq
  L_1^\n(t,a) = \frac{1}{8}\left[\frac{5a}{3}\left(t^2-a^2\right)^{3/2}
				+a^3\,\sqrt{t^2-a^2}+t^4\arcsin\frac{a}{t}\right].
\eeq
Further,
\begin{align*}
  L_2^\n(t,a)
	= {} & \int tL_2^*(t,a)\,\dd t = \int t\sqrt{t^2-a^2}\;\dd t+a\int t\arcsin\frac{a}{t}\;\dd t\\
	= {} & \frac{1}{3}\left(t^2-a^2\right)^{3/2}+\frac{a}{2}\left(a\sqrt{t^2-a^2}
						+t^2\,\arcsin\frac{a}{t}\right)
\end{align*}
(see the calculations of $L_1^\n(t,a)$ and $L_1^*(t,a)$).  
\end{proof}

\section{Two triangles}
\vspace{-0.15cm}
Now considering a rectangle $\R$ of side lengths $a$ and $b$ ($a\leq b$) as made up of two congruent copies of the triangle $\T$ with different orientation and sharing their hypotenuses. Let $g_2$ and $G_2$ denote respectively the density function and the distribution function of the distance between two uniformly and independently distributed points $P_1$ and $P_2$, one in each of the triangles.

Using the {\em indirect method} (see Ghosh\, \cite[pp.\ 22/23]{Ghosh}), it is possible to calculate $g_2$ from the density function $\gR$ of the distance between two random points inside $\R$  and the density function $g$ of $\T$ (according to Theorem~\ref{distance-density}). The probability of $P_1$ and $P_2$ belonging to the same triangle is equal to the probability of $P_1$ and $P_2$ belonging to different triangles, each probability being $1/2$. Therefore,
\beq
\vspace{-0.01cm}
  \gR = \frac{1}{2}\,(g+g_2)\,, \quad\mbox{whence}\quad g_2 = 2\,\gR-g \quad\mbox{and}\quad G_2 = 2\,\GR-G\,, 
\vspace{-0.01cm}
\eeq
where $\GR$ denotes the corresponding distribution function of $\gR$. As already mentioned, $\gR$ was found by Ghosh \cite{Ghosh} (Theorem\ 2, pp.\ 18/19). Denoting by $\AR=ab$ the area of $\R$ and by $\uR=2(a+b)$ its perimeter, we can write Ghosh's result as
\beq
  \gR(t) = \left\{\begin{array}{ll}
	\dfrac{2t}{\AR^2}\,\JR^*(t) & \mbox{if} \quad 0\leq t\leq c=\sqrt{a^2+b^2}\,,\\[0.4cm]
	0 & \mbox{else}\,,
  \end{array}\right.
\eeq
where
\beq
  \JR^*(t) = \left\{\begin{array}{lcl}
	H_{\mathcal{R},\;\!1}^*(t) & \mbox{if} & 0\leq t<a\,,\\[0.1cm]
	H_{\mathcal{R},\;\!2}^*(t) & \mbox{if} & a\leq t<b\,,\\[0.1cm]
	H_{\mathcal{R},\;\!3}^*(t) & \mbox{if} & b\leq t\leq c
  \end{array}\right.
\eeq
with
\begin{align*}
  H_{\mathcal{R},\;\!1}^*(t) = {} & \pi\AR-\uR t+t^2\,,\\
  H_{\mathcal{R},\;\!2}^*(t) = {} & \!-\!a^2-2bt+2bL_2^*(t,a)\,,\\
  H_{\mathcal{R},\;\!3}^*(t) = {} & \!-\!(\pi\AR+c^2)-t^2+2bL_2^*(t,a)+2aL_2^*(t,b)
\end{align*}
($L_2^*$ as in our Theorem \ref{distance-density}). From this, one easily concludes the distribution function
\beq
  \GR(t) = \int_{-\infty}^t\gR(\tau)\,\dd\tau = \left\{\begin{array}{lcl}
	0 & \mbox{if} & t<0\,,\\[0.1cm]
	\dfrac{2}{\AR^2}\,\JR^\n(t) & \mbox{if} & 0\leq t<c\,,\\[0.35cm]
	1 & \mbox{if} & t\geq c\,,
  \end{array}\right.
\eeq
where
\beq
  \JR^\n(t) = \left\{\begin{array}{lcl}
	J_{\mathcal{R},\;\!1}^\n(0,t) & \mbox{if} & 0\leq t<a\,,\\[0.1cm]
	J_{\mathcal{R},\;\!1}^\n(0,a) + J_{\mathcal{R},\;\!2}^\n(a,t) & \mbox{if} & a\leq t<b\,,\\[0.1cm]
	J_{\mathcal{R},\;\!1}^\n(0,a) + J_{\mathcal{R},\;\!2}^\n(a,b) 
		+ J_{\mathcal{R},\;\!3}^\n(b,t) & \mbox{if} & b\leq t<c
  \end{array}\right.
\eeq
with
\beq
  J_{\mathcal{R},\;\!k}^\n(s,t) = H_{\mathcal{R},\;\!k}^\n(t) - H_{\mathcal{R},\;\!k}^\n(s)\,,
\eeq
where
\beq
  \begin{array}{l@{\;=\;}l@{\;=\;}l}
	H_{\mathcal{R},\;\!1}^\n(t) & \ds{\int tH_{\mathcal{R},\;\!1}^*(t)\,\dd t} &
		\dfrac{\pi\AR t^2}{2}-\dfrac{\uR t^3}{3}+\dfrac{t^4}{4}\,,\\[0.3cm]
	H_{\mathcal{R},\;\!2}^\n(t) & \ds{\int tH_{\mathcal{R},\;\!2}^*(t)\,\dd t} &
		-\dfrac{a^2t^2}{2}-\dfrac{2bt^3}{3}+2bL_2^\n(t,a)\,,\\[0.3cm]
	H_{\mathcal{R},\;\!3}^\n(t) & \ds{\int tH_{\mathcal{R},\;\!3}^*(t)\,\dd t} &
		-\dfrac{(\pi\AR+c^2)t^2}{2}-\dfrac{t^4}{4}+2bL_2^\n(t,a)+2aL_2^\n(t,b)
  \end{array}
\eeq
with $L_2^\n$ as in Corollary \ref{G}. Now it is easy to calculate the density function~$g_2$ and the distribution function $G_2$.  Examples are to be found in subsection~\ref{Two_triangles}.

\section{Examples and simulation}
\subsection{One triangle}

Fig.\ \ref{Bild2} shows some density functions $cg$. The representation in this diagram has the advantage that $cg$ is independent from the maximum distance $c$, it depends only on the ratio $a/b$. Fig.\ \ref{Bild3} shows the comparison between distribution functions $G$ and corresponding simulation results (plot markers) with $c=1$ and $10^7$ pairs of random points.
\begin{figure}[h]
  \vspace{-0.1cm}
  \begin{center}
    \includegraphics[scale=1.15]{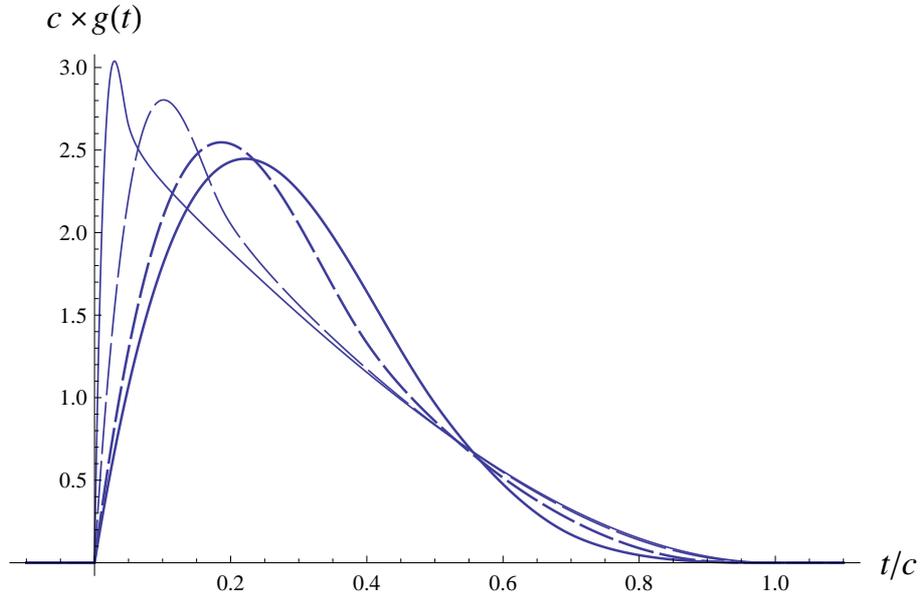}
  \end{center}
  \vspace{-0.7cm}
  \caption{\label{Bild2} $\mathcal{T}_{a,\;\!a}$ (thick), 
			$\mathcal{T}_{a,\;\!2a}$ (thick and dashed),
			$\mathcal{T}_{a,\;\!5a}$ (thin and dashed) and
			$\mathcal{T}_{a,\;\!20a}$ (thin)}
\end{figure}
\newpage
\begin{figure}  
  \vspace{0.1cm}
  \begin{center}
    \includegraphics[scale=1.15]{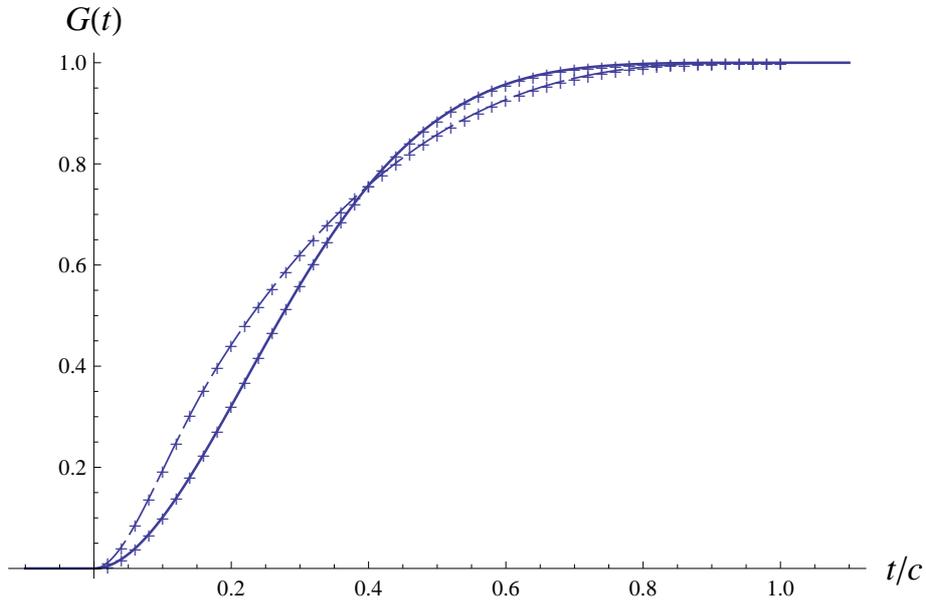}
  \end{center}
  \vspace{-0.7cm}
  \caption{\label{Bild3} $\mathcal{T}_{a,\;\!a}$ (thick) and 
			$\mathcal{T}_{a,\;\!5a}$ (thin and dashed)}
  \vspace{0cm}
\end{figure}

\subsection{Two triangles} \label{Two_triangles}

Fig.\ \ref{Bild4} shows the density function $cg_2$ for $b=5a$, and the corresponding functions $c\gR$ and $cg$. Fig.\ \ref{Bild5} shows the distribution function $G_2$ for $b=5a$, and the corresponding functions $\GR$ and $G$. The plot markers result from a simulation with $c=1$ and $10^7$ pairs of random points.
\begin{figure}[h]
  \vspace{-0.1cm}
  \begin{center}
    \includegraphics[scale=1.165]{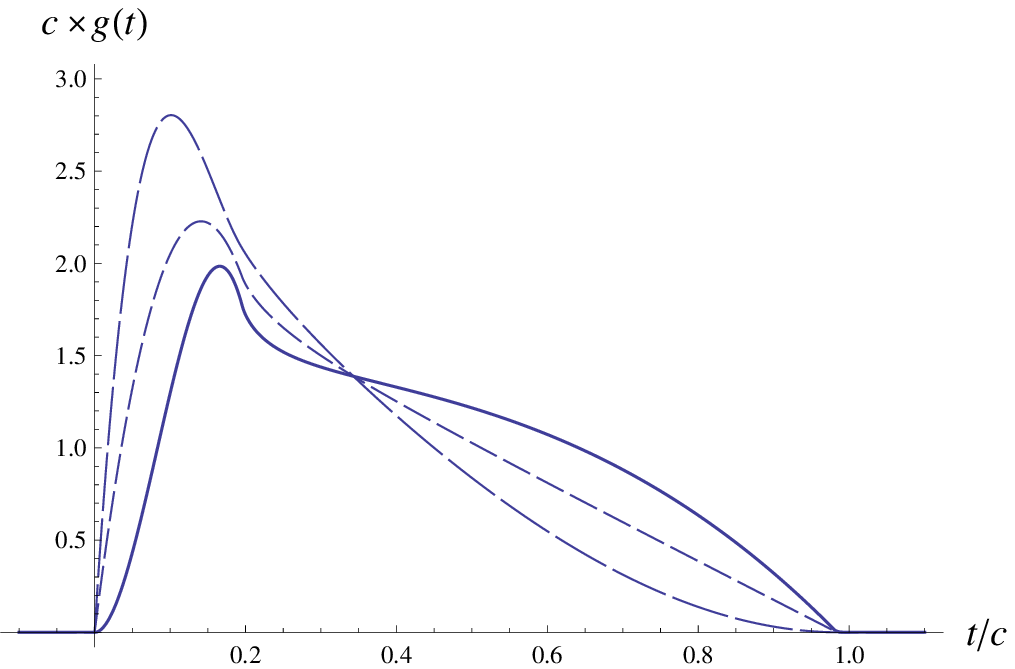}
  \end{center}
  \vspace{-0.7cm}
  \caption{\label{Bild4} $c\gR$ (short-dashed), $cg$ (long-dashed) and $cg_2$ for $b=5a$}
  \vspace{-1.7cm}
\end{figure}
\newpage

\begin{figure}[h]
  \vspace{-0.1cm}
  \begin{center}
    \includegraphics[scale=1.165]{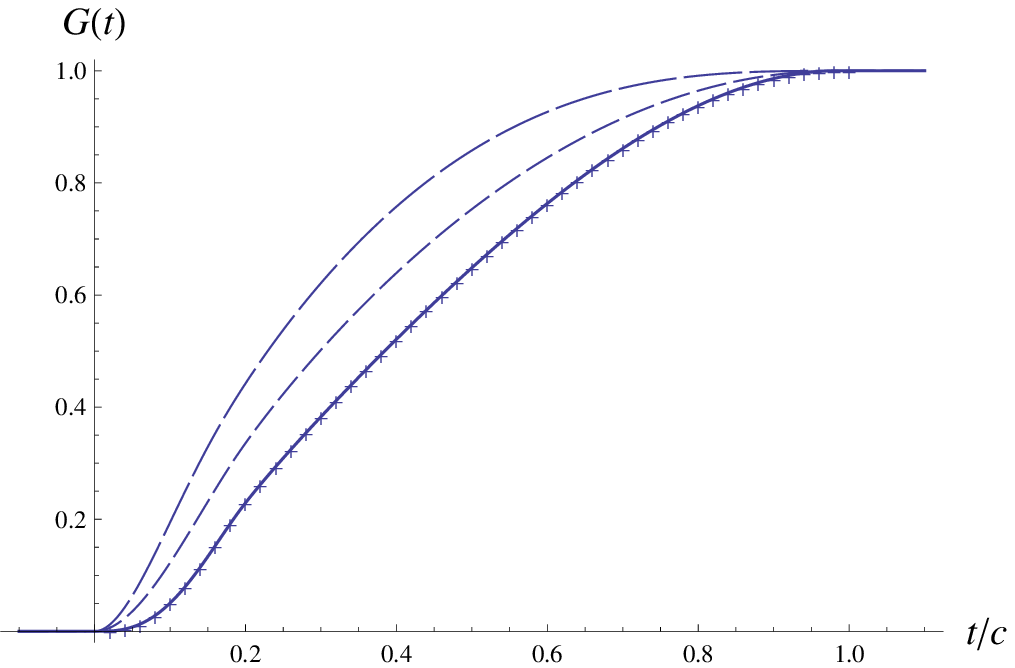}
  \end{center}
  \vspace{-0.7cm}
  \caption{\label{Bild5} $\GR$ (short-dashed), $G$ (long-dashed) and $G_2$ for $b=5a$}
  \vspace{0cm}
\end{figure}

 \vspace{1cm}
 \begin{center} 
 Uwe B\"asel\\[0.2cm] 
 HTWK Leipzig, University of Applied Sciences,\\
 Faculty of Mechanical and Energy Engineering,\\
 PF 30 11 66, 04251 Leipzig, Germany,\\[0.2cm]
 \end{center}

\begin{thebibliography}{11}

\bibitem{AMS}
Ambartzumjan, R. V.; Mecke, J.; Stoyan, D.:
{\em Geometrische Wahrscheinlichkeiten und Stochastische Geometrie,}
Akademie Verlag, Berlin, 1993.

\bibitem{Baesel}
B\"asel, U.:
Random chords and point distances in regular polygons,
arXiv: 1204.2707v2 [math.PR] 5 Sep 2012.

\bibitem{Borel}
Borel, E.:
{\em Principes et Formules Classiques du Calcul des Probabilit\'es,}
Gauthier-Villars, Paris, 1925. 

\bibitem{Bronstein}
Bronstein, I. N.; Semendjajew, K. A.:
{\em Taschenbuch der Mathematik,}
Verlag Nauka, Moskau und BSB\ B.\ G.\ Teubner Verlagsgesellschaft, Leipzig, 1989.

\bibitem{Gates}
Gates, J.:
Some properties of chord length distributions,
{\em J. Appl. Prob.,} {\bf 24} (1987), 863-874.

\bibitem{Ghosh}
Ghosh, B.:
Random distances within a rectangle and between two rectangles,
{\em Bull. Calcutta Math. Soc.,} {\bf 43} (1951), 17-24.  

\bibitem{Gille}
Gille, W.:
The chord length distribution of parallelepipeds with their limiting cases,
{\em Exp. Techn. Phys.,} \textbf{36} (1988), 197-208.

\bibitem{HO}
Harutyunyan, H. S.; Ohanyan, V. K.:
The chord length distribution function for regular polygons,
{\em Adv. Appl. Prob. (SGSA),} {\bf 41} (2009), 358-366.

\bibitem{Marsaglia}
Marsaglia, G.; Narasimhan, B. G.; Zaman, A.:
The distance between random points in rectangles,
{\em Commun. Statist. - Theory Meth.,} \textbf{19(11)} (1990), 4199-4212.

\bibitem{Mathai}
Mathai, A. M.: 
{\em An Introduction to Geometrical Probability,}
Gordon and Breach, Australia, 1999.

\bibitem{MMP}
Mathai, A. M.; Moschopoulos, P.; Pederzoli, G.:
Random points associated with rectangles,
{\em Rend. Circ. Mat. Palermo,} Serie II, Tomo XLVIII (1999), 163-190.

\bibitem{Piefke}
Piefke, F.:
Beziehungen zwischen der Sehnenl\"angenverteilung und der Verteilung des Abstandes zweier zuf\"alliger Punkte im Eik\"orper,
{\em Z.~Wahrscheinlichkeitstheorie verw. Gebiete,} \textbf{43} (1978), 129-134. 

\bibitem{Santalo}
Santal\'o, L.\  A.:
{\em Integral Geometry and Geometric Probability,}
Addison-Wesley, London, 1976.

\bibitem{Sulanke}
Sulanke, R.:
Die Verteilung der Sehnenl\"angen an ebenen und r\"aumlichen Figuren,
{\em Math. Nachr.,} \textbf{23} (1961), 51-74.

\bibitem{ZP}
Zhuang, Y.; Pan, J.:
A geometrical probability approach to location-critical network performance metrics,
{\em INFOCOM, Proceedings IEEE,} (2012), 1817-1825.

\end{thebibliography}
\end{document}